\documentclass{article}
\usepackage[english]{babel}
\usepackage{geometry}
\usepackage{amsmath,amsthm,amssymb,enumerate,latexsym}

\usepackage[colorlinks=true,linkcolor=blue,citecolor=blue,urlcolor=black,pdfpagelabels=false]{hyperref}

\usepackage{thmtools}
\theoremstyle{plain}
  \declaretheorem[numberwithin=section]{theorem}
  \declaretheorem[numberlike=theorem]{corollary}
  
  \declaretheorem[numberlike=theorem]{lemma}

\theoremstyle{definition}
  
  \declaretheorem[numberlike=theorem]{example}
  \declaretheorem[numberlike=theorem]{remark}
\newenvironment{acknowledgements}{\bigskip\textbf{Acknowledgements.}}{}

\newcommand{\email}[1]{{\textit{Email:} \texttt{#1}}}

\newcommand{\assign}{:=}

\newenvironment{enumerateroman}{\begin{enumerate}[(i)] }{\end{enumerate}}

\begin{document}

\title{Refined counting of core partitions into $d$-distinct parts}

\author{
  Hannah E.~Burson
  \thanks{University of Illinois at Urbana-Champaign, \email{hburso2@illinois.edu}}
  \and
  Simone Sisneros-Thiry
  \thanks{University of Illinois at Urbana-Champaign, \email{thiry2@illinois.edu}}
  \and
  Armin Straub
  \thanks{University of South Alabama, \email{straub@southalabama.edu}}
}

\date{October 6, 2019}

\maketitle

\begin{abstract}
  Using a combinatorial bijection with certain abaci diagrams, Nath and
  Sellers have enumerated $(s, m s \pm 1)$-core partitions into distinct
  parts. We generalize their result in several directions by including the
  number of parts of these partitions, by considering $d$-distinct partitions,
  and by allowing more general $(s, m s \pm r)$-core partitions. As an
  application of our approach, we obtain the average and maximum number of
  parts of these core partitions.
\end{abstract}

\section{Introduction}

If $\lambda_1 \geq \lambda_2 \geq \ldots \geq \lambda_r
\geq 1$, where the parts $\lambda_i$ are integers, then $\lambda =
(\lambda_1, \lambda_2, \ldots, \lambda_r)$ is an (integer) partition of $|
\lambda | = \lambda_1 + \ldots + \lambda_r$. Its number of parts is $n
(\lambda) = r$. Excellent introductions to partitions include
\cite{andrews-part} and \cite{ae-partitions}. A partition is $d$-distinct
if its parts differ by at least $d$. In the case $d = 1$, these are precisely
partitions into distinct parts. The Young diagram of a partition $\lambda$ is
a left-justified array of square cells, where the first row contains
$\lambda_1$ cells, the second row $\lambda_2$ cells, and so on. The hook
assigned to each cell $u$ consists of the cell $u$ itself as well as all cells
below $u$ and to the right of $u$. The length of a hook is the number of cells
it contains. A partition $\lambda$ is said to be {\emph{$t$-core}} if there is
no hook in its Young diagram that has length $t$. For an introduction to this
notion we refer to, for instance, \cite{ahj-cores}. More generally,
$\lambda$ is said to be $(t_1, t_2, \ldots, t_m)$-core if $\lambda$ is
$t_i$-core for all $i \in \{ 1, 2, \ldots, m \}$.

In this paper, we are concerned with counting certain core partitions. This
line of inquiry has received increasing interest since Anderson
\cite{anderson-cores} proved that the number of $(s, t)$-core partitions is
\begin{equation}
  \frac{1}{s + t}  \binom{s + t}{s} \label{eq:anderson}
\end{equation}
if $s$ and $t$ are coprime (otherwise, there are infinitely many such
partitions). It remains an open problem to similarly enumerate $(s, t)$-core
partitions into distinct parts. Towards that problem, it was shown in
\cite{s-corepartitions} that $(s, m s - 1)$-core partitions into distinct
parts are counted by Fibonacci-like numbers. This count was generalized by
Nath and Sellers \cite{ns-abaci} to also include $(s, m s + 1)$-core
partitions. Following \cite{ns-abaci}, $(s, m s \pm 1)$-core partitions
refer to $(s, m s - 1)$-core or $(s, m s + 1)$-core partitions.

\begin{theorem}
  \label{thm:cores:ns}Let $m, s \geq 1$. The number $C_{s, m}^{\pm}$ of
  $(s, m s \pm 1)$-core partitions into distinct parts is characterized by
  \begin{equation*}
    C_{s, m}^{\pm} = C_{s - 1, m}^{\pm} + m C_{s - 2, m}^{\pm},
  \end{equation*}
  for $s \geq 3$, with the initial conditions $C_{1, m}^{\pm} = 1$,
  $C_{2, m}^- = m$ and $C_{2, m}^+ = m + 1$. 
\end{theorem}

In the case $m = 1$, the numbers $C_{s, 1}^- = F_s$ are the familiar Fibonacci
numbers. This special case was predicted by Amdeberhan
\cite{amdeberhan-conj} and also independently proved by Xiong
\cite{xiong-cores}.

Nath and Sellers \cite{ns-abaci} proved (the case of $(s, m s + 1)$-core
partitions of) Theorem~\ref{thm:cores:ns} combinatorially by viewing the
partitions as certain abaci. This correspondence will be introduced in
Section~\ref{sec:abaci}. Building on their combinatorial approach, our main
result extends Theorem~\ref{thm:cores:ns} in various directions:
\begin{enumerate}
  \item we keep track of the number of parts of the involved partitions,
  
  \item we consider partitions into $d$-distinct parts (which, in the case $d
  = 1$, are partitions into distinct parts),
  
  \item we allow $(s, m s + r)$-core partitions for certain integers $r$, and
  
  \item we consider $s$-core partitions with maximum hook length $m s + r$ for
  all $m, r \in \mathbb{Z}$.
\end{enumerate}
To that end, let $\mathcal{C}_{s, m, r}^d$ be the set of $(s, m s + r)$-core
partitions into $d$-distinct parts. We denote with
\begin{equation}
  \boldsymbol{C}_{s, m, r}^d (q) = \sum_{\lambda \in \mathcal{C}_{s, m, r}^d}
  q^{n (\lambda)} \label{eq:C:def}
\end{equation}
the generating polynomial for the number of parts of the partitions in
$\mathcal{C}_{s, m, r}^d$. If $d$ is omitted in the notation, it is implicit
that $d = 1$. Observe that $C_{s, m}^{\pm} =\boldsymbol{C}_{s, m, \pm 1} (1)$.

\begin{theorem}
  \label{thm:cores:main}Let $d, m, r \geq 1$. Write $f_s (q)
  =\boldsymbol{C}_{s, m, r}^d (q)$. If $s > d + 1$ and $r \leq d$, then
  \begin{equation*}
    f_s (q) = f_{s - 1} (q) + (q + q^2 + \ldots + q^m) f_{s - d - 1} (q) .
  \end{equation*}
  If $r = - 1$, then the same conclusion holds provided that $s > 2$, if $d =
  1$, and $s > 2 d + 1$, if $d > 1$.
\end{theorem}

We provide initial conditions, which specialize to the ones in
Theorem~\ref{thm:cores:ns}, in Lemma~\ref{lem:abaci:initial}. In light of
Lemma~\ref{lem:cores:abaci:smallr}, we prove Theorem~\ref{thm:cores:main} as a
special case of our Theorems~\ref{thm:abaci:rec} and \ref{thm:abaci:rec:neg}
(as well as Lemma~\ref{lem:abaci:rec:neg:r1}, in the case $r = - 1$ and $s = 2
d + 2$, and Lemma~\ref{lem:abaci:rec:neg:d1} in the case $r = - 1$ and $d =
1$) on certain abaci. In fact, it follows from Theorem~\ref{thm:abaci:rec}
that the condition $r \leq d$ can be dropped if $\boldsymbol{C}_{s, m, r}^d
(q)$ is modified to be the generating polynomial of parts of $s$-core
partitions into $d$-distinct parts with largest hook of length less than $m s
+ r$. Additionally, in Theorem~\ref{thm:abaci:rec:neg}, we provide an
extension to negative values $r < - 1$.

Theorem~\ref{thm:cores:main} unifies several results in the recent literature:
as indicated, the special cases $(d, r, q) = (1, \pm 1, 1)$ result in
Theorem~\ref{thm:cores:ns} by Straub \cite{s-corepartitions} and Nath and
Sellers \cite{ns-abaci}. More generally, the case $(d, r) = (1, 1)$ has been
established by Paramonov \cite{paramonov-core}, who notes that the approach
applies to the case $(d, r) = (1, - 1)$ as well. In the cases $(m, q) = (1,
1)$, we obtain results on $d$-distinct partitions by Sahin
\cite{sahin-cores}, which were generalized by Kravitz \cite{kravitz-cores}
who also considers the case $m > 1$.

In Theorems~\ref{thm:abaci:rec} and \ref{thm:abaci:rec:neg}, we show that the
enumeration in Theorem~\ref{thm:cores:main} extends naturally to $s$-core
abaci of bounded height. In Section~\ref{sec:abaci}, we review the
combinatorial correspondence between core partitions and abaci, and prove
several preliminary results. Let us note, for instance, that Nath and Sellers
\cite{ns-abaci} prove the case $(d, r, q) = (1, 1, 1)$ in a combinatorial
manner and then use an algebraic argument based on generating functions to
deduce the case $(d, r, q) = (1, - 1, 1)$. At the end of
Section~\ref{sec:abaci}, we demonstrate that this algebraic argument can be
replaced (and naturally generalized) by a combinatorial observation on abaci.

Section~\ref{sec:proof} is concerned with enumerating $s$-core abaci of
bounded height. In particular, by looking at $s$-core abaci of bounded height,
we provide a proof of Theorem~\ref{thm:cores:main}. We include applications of
our approach in Section~\ref{sec:max} by determining the maximum number of
parts of certain core partitions and by enumerating core partitions with
maximal initial gaps. As another application, we indicate in
Section~\ref{sec:avg} how to determine the average number of parts (and,
likewise, higher moments) of the core partitions studied here.

\section{Preliminaries and $s$-core abaci}\label{sec:abaci}

In this section, we introduce abaci and discuss their relation to core
partitions. It is then shown in the next section that
Theorem~\ref{thm:cores:main} on enumerating $(s, m s + r)$-core partitions is
an instance of, and more naturally expressed as, a more general result on
$s$-core abaci of bounded height.

An $s$-abacus is an array consisting of $s$ columns, labelled $0, 1, 2,
\ldots, s - 1$, and some number of rows, labelled $0, 1, 2, \ldots$, where
each entry is either occupied by a bead or unoccupied (a spacer). The entry in
row $i$ and column $j$ is said to be in position $i s + j$. Placing a set of
nonnegative integers $\{ a_1, a_2, \ldots \}$ on an $s$-abacus means to
construct the $s$-abacus with beads precisely in the positions $a_1, a_2,
\ldots$. The $s$-abacus corresponding to a partition $\lambda$ is obtained by
placing the set of its first column hook lengths on an $s$-abacus.
Consequently, the number of parts of $\lambda$ equals the number of beads of
the corresponding abacus $A$. We write $n (A)$ for the number of beads in $A$.
We say that an $s$-abacus $A$ has spacing $d$ if it corresponds to a partition
with parts that differ by at least $d$. Equivalently, the positions of beads
in $A$ differ by more than $d$.

In the sequel, an $s$-core abacus refers to an $s$-abacus corresponding to an
$s$-core partition. By \cite[Lemma~7]{ns-abaci}, which is equivalent to
\cite[Lemma~2.1]{xiong-cores}, these are characterized as follows.

\begin{lemma}[{\cite[Lemma~7]{ns-abaci}}]
  \label{lem:coreabacus}An $s$-abacus corresponds to an $s$-core partition if
  and only if the first column is empty and no spacers occur below a bead.
\end{lemma}

\begin{example}
  The partition $\lambda = (7, 5, 4, 2, 1)$ has first column hook lengths $\{
  1, 3, 6, 8, 11 \}$. Its $5$-abacus and $7$-abacus are:
  
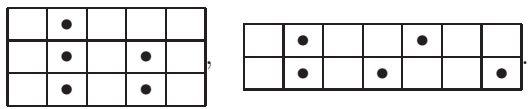
\begin{equation*}
    \begin{array}{|c|c|c|c|c|}
       \hline
       \phantom{\bullet} & \bullet & \phantom{\bullet} & \phantom{\bullet} & \phantom{\bullet} \\
       \hline
       & \bullet &  & \bullet & \\
       \hline
       & \bullet &  & \bullet & \\
       \hline
     \end{array}, \quad
     \begin{array}{|c|c|c|c|c|c|c|}
       \hline
       \phantom{\bullet} & \bullet & \phantom{\bullet} & \phantom{\bullet} & \bullet & \phantom{\bullet} & \phantom{\bullet} \\
       \hline
       & \bullet &  & \bullet &  &  & \bullet\\
       \hline
     \end{array} .
  \end{equation*}
  In light of Lemma~\ref{lem:coreabacus}, we see that $\lambda$ is not a
  $7$-core partition (because there is one bead with a spacer below it) but is
  a $5$-core partition (because the first column is empty and no spacers occur
  below a bead). Because $\lambda$ has distinct parts, both abaci have spacing
  $1$.
\end{example}

Let $\mathcal{A}_{s, m}^d$ be the set of $s$-core abaci with (at most) $m$
rows and spacing $d$ (throughout, $m \geq 1$ and $d \geq 1$). More
generally, for integers $r$ (possibly negative), let $\mathcal{A}_{s, m, r}^d$
denote the set of $s$-core abaci with spacing $d$ such that the maximum position
is (strictly) less than $m s + r$. Note that, by definition, $\mathcal{A}_{s, m}^d
=\mathcal{A}_{s, m, 0}^d$. (Since the first column of an $s$-core abacus is
necessarily empty, we also have $\mathcal{A}_{s, m}^d =\mathcal{A}_{s, m,
1}^d$.) In the sequel, if $d = 1$, we may omit $d$ in the notation. That is,
$\mathcal{A}_{s, m} \assign \mathcal{A}_{s, m}^1$ and $\mathcal{A}_{s, m, r}
\assign \mathcal{A}_{s, m, r}^1$.

Recall that $n (A)$ is the number of beads in the abacus $A$. Let
\begin{equation}
  \boldsymbol{A}_{s, m, r}^d (q) = \sum_{A \in \mathcal{A}_{s, m, r}^d} q^{n
  (A)} \label{eq:A:def}
\end{equation}
be the generating polynomial for the number of beads in the abaci in
$\mathcal{A}_{s, m, r}^d$. As for $\mathcal{A}_{s, m, r}^d$, we drop $d$ and
$r$ from the notation if $d = 1$ or $r = 0$, respectively. Note that
$\boldsymbol{A}_{0, m, r}^d (q) = 1$, representing the empty abacus. Since it is
natural and convenient for certain results, we set $\boldsymbol{A}_{s, m, r}^d
(q) = 0$ for $s < 0$.

Using the correspondence with abaci we can easily show the following result,
which, in the case $d = 1$, was proved in \cite[Lemma~2.2]{s-corepartitions}
and \cite[Lemma~3.2]{xiong-cores} and which is equivalent to
\cite[Lemma~2.5]{kravitz-cores}.

\begin{lemma}
  \label{lem:cores:maxhook}If $1 \leq r \leq d$, then every $(s, s +
  r)$-core partition into $d$-distinct parts has maximum hook length $< s +
  r$.
\end{lemma}

\begin{proof}
  Let $A$ be the $s$-abacus of an $(s, s + r)$-core partition $\lambda$ into
  $d$-distinct parts. We need to show that all beads of $A$ (recall that these
  correspond to the first column hook lengths of $\lambda$) are in positions
  with labels $< s + r$. Assume for the sake of contradiction that $A$ has a
  bead in position $x \geq s + r$. By Lemma~\ref{lem:coreabacus}, because
  $\lambda$ is an $s$-core partition, there must also be a bead in position $x
  - s$. Likewise, because $\lambda$ is a $(s + r)$-core partition, there must
  be a bead in position $x - (s + r)$. However, $| (x - s) - (x - (s + r)) | =
  r \leq d$, which contradicts the requirement that positions of beads in
  $A$ differ by more than $d$.
\end{proof}

By construction, the abaci in $\mathcal{A}_{s, m, r}^d$ correspond to $s$-core
partitions into $d$-distinct parts with largest hook $< m s + r$.
Equivalently, these are $(s, m s + r, m s + r + 1)$-core partitions into
$d$-distinct parts (note that, by Lemma~\ref{lem:cores:maxhook} applied with
$d = r = 1$, a $(t, t + 1)$-core partition into distinct parts has largest
hook less than $t$). If $r \leq d$, this simplifies as follows.

\begin{lemma}
  \label{lem:cores:abaci:smallr}If $1 \leq r \leq d$ or if $r = -
  1$, then the abaci in $\mathcal{A}_{s, m, r}^d$ are in $1$-$1$
  correspondence with $(s, m s + r)$-core partitions into $d$-distinct parts.
\end{lemma}

\begin{proof}
  We need to show that $(s, m s + r)$-core partitions into $d$-distinct parts
  have maximum hook length $< m s + r$. In the case $1 \leq r \leq
  d$, applying Lemma~\ref{lem:cores:maxhook} with $m s$ in place of $s$, we
  find that every $(s, m s + r)$-core partition into $d$-distinct parts indeed
  has maximum hook length $< m s + r$.
  
  Finally, consider the case $r < 0$. Note that every $(s, m s + r)$-core
  partition is $(m s + r, m s)$-core. Hence, if $d \geq | r |$, then
  Lemma~\ref{lem:cores:maxhook} shows that every $(s, m s + r)$-core partition
  into $d$-distinct parts has maximum hook length $< m s$. If $r = - 1$, it
  follows that the maximum hook length is $< m s - 1$.
\end{proof}

\begin{example}
  Lemma~\ref{lem:cores:abaci:smallr} does not generally hold for other values
  of $r$. For instance, consider the case $r = - 2$ and $(s, m) = (6, 1)$. The
  partition $\lambda = (4, 1)$ is $(s, m s + r) = (6, 4)$-core and
  $3$-distinct. However, its maximum hook length is $5 \nless m s + r = 4$.
\end{example}

\begin{lemma}
  \label{lem:removecols}Let $s > d + 1$. Removing the last $d + 1$ columns,
  not all empty, from an $s$-core abacus with spacing $d$ results in an $(s -
  d - 1)$-core abacus with spacing $d$.
\end{lemma}

\begin{proof}
  In light of Lemma~\ref{lem:coreabacus}, the result is clearly an $(s - d -
  1)$-core abacus. It remains to observe that the reduced abacus has spacing
  $d$. For that, it suffices to check that, in the reduced abacus, each bead
  in the first row is followed by $d$ empty positions. Since the initial
  abacus had spacing $d$, we need only consider the last such bead. In the
  initial abacus, this bead was followed by $d_1 \geq d$ spacers, then a
  bead (in the last $d + 1$ columns), followed by another $d_2 \geq d$
  spacers. After removal of the last $d + 1$ columns, it is therefore followed
  by $d_1 + d_2 - d \geq d$ spacers.
\end{proof}

\begin{example}
  We note that this result is not true for removing the first $d + 1$ columns.
  This is illustrated, for instance, with $d = 2$, by the $5$-core abacus with
  beads in positions $1, 4, 9$.
\end{example}

Nath and Sellers \cite{ns-abaci} prove Theorem~\ref{thm:cores:ns} by a
combinatorial argument for $(s, m s + 1)$-core partitions. They then deduce
the case of $(s, m s - 1)$-core partitions by proving the ($q = 1$ case of
the) relation
\begin{equation}
  \boldsymbol{A}_{s, m, - 1} (q) =\boldsymbol{A}_{s - 1, m} (q) + (q + q^2 +
  \ldots + q^{m - 1}) \boldsymbol{A}_{s - 2, m} (q) \label{eq:A:ns:minus}
\end{equation}
using an algebraic argument based on generating functions. Note that, by
Theorem~\ref{thm:abaci:rec}, the relation \eqref{eq:A:ns:minus} is equivalent
to
\begin{equation*}
  \boldsymbol{A}_{s, m, - 1} (q) =\boldsymbol{A}_{s, m} (q) - q^m \boldsymbol{A}_{s
   - 2, m} (q) .
\end{equation*}
Our next result provides a simple combinatorial proof \ of this identity and,
further, generalizes it to $\boldsymbol{A}_{s, m, - r} (q)$ for any $r \geq
1$. For the case $s = r$ in \eqref{eq:A:rneg:rec}, it is understood that
$\boldsymbol{A}_{s, m, r}^d (q) = 0$ if $s < 0$.

\begin{lemma}
  \label{lem:A:rneg:rec}Let $m, r \geq 1$. If $s \geq r$, then
  \begin{equation}
    \boldsymbol{A}_{s, m} (q) =\boldsymbol{A}_{s, m, - r} (q) + q^m \sum_{k = 1}^r
    \boldsymbol{A}_{s - k - 1, m} (q) \boldsymbol{A}_{k - 1, m - 1} (q) .
    \label{eq:A:rneg:rec}
  \end{equation}
\end{lemma}

\begin{proof}
  Suppose $s > r$. Observe that $\boldsymbol{A}_{s, m} (q) -\boldsymbol{A}_{s, m,
  - r} (q)$ consists of those $s$-core abaci with spacing $d = 1$, which have
  $m$ rows and last bead in position $m s - k$ for $k = 1, 2, \ldots, r$. Fix
  one of these values for $k$, corresponding to one of the last $r$ columns.
  That column always contains $m$ beads, contributing $q^m$ to
  \eqref{eq:A:rneg:rec}. The $k - 1$ columns after that column form an abacus
  in $\mathcal{A}_{k - 1, m - 1}$, while the first $s - k - 1$ columns form an
  abacus in $\mathcal{A}_{s - k - 1, m}$. The one remaining column is the
  empty column preceding the column with $m$ beads.
  
  If $s = r$, the same argument still applies but $k = r$ is not possible.
  Since $\boldsymbol{A}_{- 1, m} (q) = 0$, the summand corresponding to $k = r$
  is zero, so that equation \eqref{eq:A:rneg:rec} still holds.
\end{proof}

\section{$s$-core abaci of bounded height}\label{sec:proof}

As introduced in Section~\ref{sec:abaci}, $\mathcal{A}_{s, m, r}^d$ is the set
of $s$-core abaci with spacing $d$ such that the maximum position is less than
$m s + r$. Recall that $n (A)$ is the number of beads in $A$.

As a consequence of Lemma~\ref{lem:cores:abaci:smallr}, if $1 \leq r
\leq d$ or if $r = - 1$, then
\begin{equation}
  \boldsymbol{C}_{s, m, r}^d (q) =\boldsymbol{A}_{s, m, r}^d (q), \label{eq:CA}
\end{equation}
where $\boldsymbol{C}_{s, m, r}^d (q)$, defined in \eqref{eq:C:def}, is the
generating polynomial for the number of parts of $(s, m s + r)$-core
partitions into $d$-distinct parts. In particular,
Theorem~\ref{thm:cores:main} on core partitions is contained in the following
result on abaci.

\begin{theorem}
  \label{thm:abaci:rec}Let $d, m \geq 1$, $r \geq 0$, and write $f_s
  (q) =\boldsymbol{A}_{s, m, r}^d (q)$. If $s > d + 1$ and $s > r$, then
  \begin{equation}
    f_s (q) = f_{s - 1} (q) + (q + q^2 + \ldots + q^m) f_{s - d - 1} (q) .
    \label{eq:abaci:rec}
  \end{equation}
\end{theorem}

We note that, in the case $d = 1$, the condition $s > d + 1$ may be dropped,
because \eqref{eq:abaci:rec} trivially holds for $s = 1$ and $s = 2$ (with the
understanding that $f_s (q) = 0$ if $s < 0$).

\begin{proof}
  By definition, $f_s (q)$ is the generating polynomial for $s$-core abaci
  with spacing at least $d$, which consist of $m + 1$ rows such that the top
  row has beads only in the first $r$ positions. Let $g$ be the size of the
  gap between the last bead in the first row and the first bead in the second
  row; or $g = \infty$ if there are no beads in the second row. By assumption,
  $g \geq d$.
  
  There are two disjoint possibilities for such abaci, depending on whether $g
  > d$ and the last column is empty (case \ref{i:gap1}), or not (case
  \ref{i:gap2}).
  \begin{enumerate}
    \item \label{i:gap1}Suppose that $g > d$ and that the last column is
    empty. Since $r \leq s - 1$, by removing the last column, we see that
    these abaci are in bijective correspondence with abaci in $\mathcal{A}_{s
    - 1, m, r}^d$. The generating polynomial for the latter is $f_{s - 1}
    (q)$.
    
    \item \label{i:gap2}Suppose that $g = d$ or that the last column is
    non-empty. In that case, exactly one of the last $d + 1$ columns contains
    beads. We must again consider two cases based on whether the last bead
    (that is, the bead in the highest position) is located in the last $d + 1$
    columns of the ($m + 1$)st row (case \ref{enm:s:small}) or not (case
    \ref{enm:s:big}). (We note that, if $s > r + d + 1$, then we are
    necessarily in case \ref{enm:s:big}.)
    \begin{enumerateroman}
      \item \label{enm:s:big}Suppose the last bead is not located in the last
      $d + 1$ columns of the $(m + 1)$-th row (that is, its position is $< (m
      + 1) s - d - 1$). Consider the abaci obtained by removing the last $d +
      1$ columns, exactly one of which contains beads ($1, 2, \ldots, m$
      many). By Lemma~\ref{lem:removecols}, the resulting abaci are in
      $\mathcal{A}_{s - d - 1, m, r}^d$ with no bead in the ($m + 2$)nd row.
      
      To see that this correspondence is $m$ to $1$, consider an abacus $A$ in
      $\mathcal{A}_{s - d - 1, m, r}^d$ with no bead in the $m + 2$nd row. We
      need to show that there are exactly $m$ ways in which $A$ could have
      been obtained from a case \ref{enm:s:big} abacus $B \in \mathcal{A}_{s,
      m, r}^d$ by deleting the last $d + 1$ columns. Let $h \geq 1$ be
      the number of empty positions at the beginning of the second row of $A$.
      If $h \leq d$, then the last $d - h$ columns of $B$ must have been
      empty and the column before that non-empty (with $m$ choices for the
      number of beads). If $h > d$, then the last column of $B$ must have been
      non-empty (again, with $m$ choices for the number of beads). In either
      case, the remaining of the last $d + 1$ columns of $B$ must be empty. By
      construction, each of the $m$ possibilities for $B$ indeed are case
      \ref{enm:s:big} abaci from $\mathcal{A}_{s, m, r}^d$.
      
      \item \label{enm:s:small}We now are left with abaci in which the last
      bead is located in the last $d + 1$ columns of the $(m + 1)$st row.
      Equivalently, the unique non-empty column among the last $d + 1$ has $m
      + 1$ beads. Since $s > r$, this cannot be the last column. Consequently,
      $g = d$. In particular, there is a bead in the second row of the first
      non-empty column. Increase the number of beads in that first non-empty
      column to $m + 2$ (thereby adding $1, 2, \ldots, m$ beads), and remove
      the last $d + 1$ columns (thereby removing $m + 1$ beads). This process
      removes $1, 2, \ldots, m$ beads. By Lemma~\ref{lem:removecols}, the
      resulting abaci are in $\mathcal{A}_{s - d - 1, m, r}^d$ with a bead in
      the ($m + 2$)nd row.
      
      To see that this correspondence is $m$ to $1$, consider an abacus $A$ in
      $\mathcal{A}_{s - d - 1, m, r}^d$ with a bead in the ($m + 2$)nd row. We
      need to show that there are exactly $m$ ways in which $A$ could have
      been obtained from a case \ref{enm:s:small} abacus $B \in
      \mathcal{A}_{s, m, r}^d$. As in case \ref{enm:s:big}, let $h \geq
      1$ be the number of empty positions at the beginning of the second row
      of $A$. Note that column $h + 1$ of $A$ contains $m + 2$ beads. We
      necessarily have $h \leq d$, so that the last $d - h$ columns of
      $B$ must have been empty and the column before that must have contained
      $m + 1$ beads. In order to obtain $A$ from $B$, column $h + 1$ of $B$
      must have contained $2, 3, \ldots, m + 1$ beads, for a total of $m$
      choices for $B$.
    \end{enumerateroman}
  \end{enumerate}
\end{proof}

Theorem~\ref{thm:abaci:rec} has the following analog for abaci in
$\mathcal{A}_{s, m, r}^d$ where $r$ is negative. Because of
Lemma~\ref{lem:cores:abaci:smallr}, the case $r = - 1$ is of particular
interest, as those abaci correspond to $(s, m s - 1)$-core partitions into
$d$-distinct parts. As such, Theorem~\ref{thm:abaci:rec:neg} (together with
Lemma~\ref{lem:abaci:rec:neg:r1}) proves the part of
Theorem~\ref{thm:cores:main} concerned with these partitions.

\begin{theorem}
  \label{thm:abaci:rec:neg}Let $d, m \geq 1$, $r > 0$, and write $f_s (q)
  =\boldsymbol{A}_{s, m, - r}^d (q)$. If $s > 2 d + r$, then
  \begin{equation}
    f_s (q) = f_{s - 1} (q) + (q + q^2 + \ldots + q^m) f_{s - d - 1} (q) .
    \label{eq:abaci:rec:rneg}
  \end{equation}
\end{theorem}

It appears that the bound $s > 2 d + r$ is sharp for large $r$, specifically
$r > d + 1$, and $d \geq 2$. For small $r$, the bound can be improved to
$s > 2 d + 1$. We give a direct combinatorial proof, in the spirit of
Theorem~\ref{thm:abaci:rec}, of this fact in
Theorem~\ref{thm:abaci:rec:neg:rsmall}. Based on numerical evidence, the
optimal lower bounds on $s$ appear to be
\begin{equation*}
  s > 2 d + \left\{ \begin{array}{ll}
     r, & \text{if $r > d + 1$ and $d \geq 2$},\\
     1, & \text{if $r \leq d + 1$ and $d \geq 2$},\\
     r - 1, & \text{if $d = 1$} .
   \end{array} \right.
\end{equation*}
We prove the remaining case $d = 1$ in Lemma~\ref{lem:abaci:rec:neg:d1}.

\begin{proof}[Proof of Theorem~\ref{thm:abaci:rec:neg}]
  Let us prove the claim that \eqref{eq:abaci:rec:rneg} holds for all $s > 2 d
  + r$. Fix $r > 1$ and suppose, for the purposes of induction, that this
  claim is true for $r$ replaced with $1, 2, \ldots, r - 1$. In the following,
  we are going to show that the claim holds for $r$ as well. [Along the way,
  we also show that \eqref{eq:abaci:rec:rneg} holds for $r = 1$ and all $s > 2
  d + 2$. Comments and adjustments for that case are included in square
  brackets. Since \eqref{eq:abaci:rec:rneg} is proved for $r = 1$ and $s = 2 d
  + 2$ in Lemma~\ref{lem:abaci:rec:neg:r1} below, this completes the proof of
  Theorem~\ref{thm:abaci:rec:neg}.]
  
  Suppose that $s > 2 d + r$ [if $r = 1$, then $s > 2 d + 2$]. In the
  following, let $t$ be such that $t \geq s - d - 1$ (we intend to
  rewrite each of the terms $f_t (q)$ in \eqref{eq:abaci:rec:rneg}). Note that
  the condition $s > 2 d + r$ translates into $t > d + r - 1$ [while, for $r =
  1$, the condition $s > 2 d + 2$ translates into $t > d + 1$].
  
  By construction, for any integer $t \geq 0$, we have
  \begin{equation}
    \boldsymbol{A}_{t, m, - r}^d (q) =\boldsymbol{A}_{t, m - 1, t - r}^d (q) .
    \label{eq:A:rnegpos}
  \end{equation}
  Since $t \geq r$ and $t > d + 1$, we may apply
  Theorem~\ref{thm:abaci:rec} to the right-hand side of \eqref{eq:A:rnegpos}
  to obtain
  \begin{eqnarray*}
    \boldsymbol{A}_{t, m, - r}^d (q) & = & \boldsymbol{A}_{t - 1, m - 1, t - r}^d
    (q) + (q + q^2 + \ldots + q^{m - 1}) \boldsymbol{A}_{t - d - 1, m - 1, t -
    r}^d (q)\\
    & = & \boldsymbol{A}_{t - 1, m, 1 - r}^d (q) + (q + q^2 + \ldots + q^{m -
    1}) \boldsymbol{A}_{t - d - 1, m, d + 1 - r}^d (q).
  \end{eqnarray*}
  For the second equality, we applied \eqref{eq:A:rnegpos}, in reverse, to
  both summands. On the other hand, we have
  \begin{equation*}
    \boldsymbol{A}_{t, m, d + 1 - r}^d (q) =\boldsymbol{A}_{t - 1, m, d + 1 -
     r}^d (q) + (q + q^2 + \ldots + q^m) \boldsymbol{A}_{t - d - 1, m, d + 1 -
     r}^d (q),
  \end{equation*}
  as a consequence of Theorem~\ref{thm:abaci:rec}, if $d + 1 - r \geq 0$,
  or the induction hypothesis (since, if $d + 1 - r < 0$, then the condition
  $t > 2 d + (r - d - 1) = d + r - 1$ is satisfied). [If $r = 1$, then \ no
  induction hypothesis is required because we will always be in the case $d +
  1 - r \geq 0$.]
  
  Comparing the previous two equations, we conclude that, for $t \geq s -
  d - 1$,
  \begin{eqnarray}
    \boldsymbol{A}_{t, m, - r}^d (q) & = & \boldsymbol{A}_{t - 1, m, 1 - r}^d (q)
    - q^m \boldsymbol{A}_{t - d - 1, m, d + 1 - r}^d (q) \nonumber\\
    & + & \boldsymbol{A}_{t, m, d + 1 - r}^d (q) -\boldsymbol{A}_{t - 1, m, d + 1
    - r}^d (q) .  \label{eq:A:rneg:ind}
  \end{eqnarray}
  Let $g_t (q)$ be any one of the four terms on the right-hand side of
  \eqref{eq:A:rneg:ind}. As above, after checking that the conditions are
  satisfied, it follows from either Theorem~\ref{thm:abaci:rec} or the
  induction hypothesis that, for the special case $t = s$,
  \begin{equation}
    g_s (q) = g_{s - 1} (q) + (q + q^2 + \ldots + q^m) g_{s - d - 1} (q),
    \label{eq:g:rec}
  \end{equation}
  which then shows that the same relation holds for $f_s (q) =\boldsymbol{A}_{s,
  m, - r}^d (q)$, thus establishing the claimed relation
  \eqref{eq:abaci:rec:rneg}. Let us illustrate the details of checking
  \eqref{eq:g:rec} at the example of $g_s (q) =\boldsymbol{A}_{s - d - 1, m, d +
  1 - r}^d (q)$, which is the most restrictive of the four cases. In that
  case, if $d + 1 - r \geq 0$, then Theorem~\ref{thm:abaci:rec} applies
  to show \eqref{eq:g:rec} because the condition $s - d - 1 > d + 1$, that is
  $s > 2 d + 2$, follows from the assumption that $s > 2 d + r$ [whereas, in
  the case $r = 1$, we are assuming $s > 2 d + 2$ to begin with]. On the other
  hand, if $d + 1 - r < 0$ [this does not happen when $r = 1$], then the
  induction hypothesis applies to show \eqref{eq:g:rec} because the assumed
  condition $s > 2 d + r$ implies the necessary condition $s - d - 1 > 2 d +
  (r - d - 1)$.
\end{proof}

We now complete the proof of Theorem~\ref{thm:abaci:rec:neg} by showing that,
in the case $r = 1$, \eqref{eq:abaci:rec:rneg} holds for $s = 2 d + 2$.

\begin{lemma}
  \label{lem:abaci:rec:neg:r1}Let $d, m \geq 1$, and write $f_s (q)
  =\boldsymbol{A}_{s, m, - 1}^d (q)$. If $s = 2 d + 2$, then
  \eqref{eq:abaci:rec:rneg} holds.
\end{lemma}

\begin{proof}
  We need to show that
  \begin{equation}
    \boldsymbol{A}_{2 d + 2, m, - 1}^d (q) =\boldsymbol{A}_{2 d + 1, m, - 1}^d (q)
    + (q + q^2 + \ldots + q^m) \boldsymbol{A}_{d + 1, m, - 1}^d (q) .
    \label{eq:abaci:rec:neg:r1}
  \end{equation}
  Because each abacus in $\mathcal{A}_{d + 1, m, - 1}^d$ has at most one
  column of beads,
  \begin{eqnarray}
    \boldsymbol{A}_{d + 1, m, - 1}^d (q) & = & 1 + (q + q^2 + \ldots + q^{m -
    1}) + (d - 1) (q + q^2 + \ldots + q^m) \nonumber\\
    & = & (1 + (d - 1) q) (1 + q + \ldots + q^{m - 1}) . 
    \label{eq:A:small:neg:r1}
  \end{eqnarray}
  On the other hand, abaci in $\mathcal{A}_{2 d + 2, m, - 1}^d$ and
  $\mathcal{A}_{2 d + 1, m, - 1}^d$ have at most two columns of beads.
  Appending an empty column to each abacus in $\mathcal{A}_{2 d + 1, m, -
  1}^d$ yields an abacus in $\mathcal{A}_{2 d + 2, m, - 1}^d$. The abaci in
  $\mathcal{A}_{2 d + 2, m, - 1}^d$ not obtained in this fashion fall into
  three groups:
  \begin{enumerateroman}
    \item Those for which the last column is not empty (and hence contains $1,
    2, \ldots, m - 1$ many beads). The generating polynomial for these abaci
    is
    \begin{equation*}
      (q + q^2 + \ldots + q^{m - 1}) (1 + (q + q^2 + \ldots + q^m) + (d - 1)
       q),
    \end{equation*}
    based on whether there is no second column of beads (contributing the $1$
    in the second factor), or there is a column of beads including position
    $d$ (contributing $q + q^2 + \ldots + q^m$), or there is a single bead in
    positions $1, 2, \ldots, d - 1$ (contributing $(d - 1) q$).
    
    \item Those for which the last column is empty and there is a gap of
    exactly $d$ spaces between the last bead in the first row and the first
    bead in the second row. Since there are $d - 1$ possibilities for the
    location of the first column of beads (which contains $2, 3, \ldots, m$
    many beads), each of which determines the location of the second column
    (containing $1, 2, \ldots, m$ beads), the generating polynomial for these
    abaci is
    \begin{equation*}
      (d - 1) (q^2 + q^3 + \ldots + q^m) (q + q^2 + \ldots + q^m) .
    \end{equation*}
    \item Those for which the second to last column contains $m$ beads and, to
    avoid double counting, there are more than $d$ spaces following the last
    bead in the first row. The generating polynomial for these abaci is
    \begin{equation*}
      q^m (1 + (d - 1) q),
    \end{equation*}
    because, in addition to the column of $m$ beads, there can only be a
    single bead in positions $1, 2, \ldots, d - 1$.
  \end{enumerateroman}
  Adding these three generating polynomials, we have shown that
  \begin{equation*}
    \boldsymbol{A}_{2 d + 2, m, - 1}^d (q) -\boldsymbol{A}_{2 d + 1, m, - 1}^d
     (q) = q (1 + (d - 1) q) (1 + q + \ldots + q^{m - 1})^2 .
  \end{equation*}
  Together with \eqref{eq:A:small:neg:r1}, this proves
  \eqref{eq:abaci:rec:neg:r1}.
\end{proof}

The bound $s > 2 d + r$ in Theorem~\ref{thm:abaci:rec:neg} can be improved to
$s > 2 d + 1$ in the case of small $r$, specifically $r \leq d + 1$. We
give a direct combinatorial proof for that improved bound, which empirically
is best possible.

\begin{theorem}
  \label{thm:abaci:rec:neg:rsmall}Let $m \geq 1$, $1 \leq r
  \leq d + 1$, and write $f_s (q) =\boldsymbol{A}_{s, m, - r}^d (q)$. If $s
  > 2 d + 1$, then
  \begin{equation}
    f_s (q) = f_{s - 1} (q) + (q + q^2 + \ldots + q^m) f_{s - d - 1} (q) .
  \end{equation}
\end{theorem}

\begin{proof}
  Our proof is a variation of the proof of Theorem~\ref{thm:abaci:rec}. By
  definition, $f_s (q)$ is the generating polynomial for $s$-core abaci with
  spacing at least $d$ between beads, which consist of $m$ rows such that the
  top row has no beads in the last $r$ positions. As in the proof of
  Theorem~\ref{thm:abaci:rec}, let $g$ be the size of the gap between the last
  bead in the first row and the first bead in the second row. There are two
  disjoint possibilities for such abaci.
  \begin{enumerate}
    \item \label{i:neg:gap1}Suppose that $g > d$ with the last column empty
    and no bead in the final allowed position (that is, no bead in position $m
    s - r - 1$). By removing the last column, we see that these abaci are in
    bijective correspondence with abaci in $\mathcal{A}_{s - 1, m, r}^d$.
    
    \item \label{i:neg:gap2}Suppose the conditions of \ref{i:neg:gap1} are not
    satisfied. We distinguish the following cases.
    \begin{enumerateroman}
      \item Suppose that the last $d + 1$ columns are empty. This can only
      happen if $r = d + 1$ and if there is a bead in the highest allowed
      position (that is, column $s - r - 1$ has $m$ beads). We then remove the
      last $d + 1$ columns, as well as the $m$ beads from column $s - r - 1$,
      and obtain an abacus in $\mathcal{A}_{s - d - 1, m, - r}^d$.
      
      \item Otherwise, there is precisely one non-empty column among the last
      $d + 1$ columns. Suppose (additionally) that one of the $r$ columns $s -
      d - r, \ldots, s - d - 2, s - d - 1$, say column $j$, contains $m$ beads
      (equivalently, removing the last $d + 1$ columns does not result in an
      abacus in $\mathcal{A}_{s - d - 1, m, - r}^d$). In that case, the one
      non-empty column among the last $d + 1$ must be one of the last $r$ and,
      as such, contains $b < m$ beads. We then remove the last $d + 1$ columns
      and reduce the number of beads in column $j$ by $m - b$. The resulting
      abacus is in $\mathcal{A}_{s - d - 1, m, - r}^d$ and has precisely $m$
      fewer beads.
      
      \item Otherwise, we remove the last $d + 1$ columns and, by
      Lemma~\ref{lem:removecols}, obtain an abacus in $\mathcal{A}_{s - d - 1,
      m, - r}^d$.
    \end{enumerateroman}
    Fix an abacus $A \in \mathcal{A}_{s - d - 1, m, - r}^d$. To complete the
    proof, we need to show that $A$ is obtained through the described process
    from $m$ case \ref{i:neg:gap2} abaci $B \in \mathcal{A}_{s, m, - r}^d$,
    with $1, 2, \ldots, m$ additional beads. Let $h \geq 1$ be the number
    of empty positions at the beginning of the second row of $A$. If $h
    \leq d$, then let $k = s - (d - h)$. Otherwise, let $k = s$. We can
    construct $B$ from $A$ by adding $d + 1$ empty columns, and then filling
    column $k$ with $b$ beads, where $1 \leq b < m$. (Note that, if $h
    \leq d$, then $g = d$ for $B$, while otherwise $B$ has a non-empty
    last column.) It remains to similarly construct an abacus $B$ from $A$ by
    adding $d + 1$ columns and $m$ beads. For that, we distinguish three
    cases:
    \begin{itemize}
      \item If column $k$ is not among the last $r$ columns of $B$, then it
      can also be filled with $m$ beads.
      
      \item Otherwise, if one of the last $r$ columns of $A$ contains $c
      \geq 1$ beads, then we add $m - c$ beads to the corresponding
      column of $B$ and $c$ beads to column $k$ of $B$.
      
      \item Otherwise, we add $m$ beads to column $s - r - 1$ of $B$.
    \end{itemize}
    In each case, we check that abacus $B$ corresponds to case
    \ref{i:neg:gap2} and results in abacus $A$ by the process described for
    case \ref{i:neg:gap2}.
  \end{enumerate}
\end{proof}

We further provide a quick proof that Theorem~\ref{thm:abaci:rec:neg}, in the
case $d = 1$, continues to hold for all $s > r + 1$, which empirically is best
possible.

\begin{lemma}
  \label{lem:abaci:rec:neg:d1}Let $m, r \geq 1$, and write $f_s (q)
  =\boldsymbol{A}_{s, m, - r} (q)$. If $s > r + 1$, then
  \eqref{eq:abaci:rec:rneg} holds.
\end{lemma}

\begin{proof}
  By Lemma~\ref{lem:A:rneg:rec}, for $s \geq r$,
  \begin{equation*}
    \boldsymbol{A}_{s, m, - r} (q) =\boldsymbol{A}_{s, m} (q) - q^m \sum_{k =
     1}^r \boldsymbol{A}_{s - k - 1, m} (q) \boldsymbol{A}_{k - 1, m - 1} (q) .
  \end{equation*}
  Let $g_s (q)$ be any one of the terms $\boldsymbol{A}_{s, m} (q)$ or
  $\boldsymbol{A}_{s - k - 1, m} (q)$, for $k = 1, 2, \ldots, r$, on the
  right-hand side. To prove our claim, it suffices to show that, for $s
  \geq r + 2$,
  \begin{equation*}
    g_s (q) = g_{s - 1} (q) + (q + q^2 + \ldots + q^m) g_{s - 2} (q) .
  \end{equation*}
  This follows from Theorem~\ref{thm:abaci:rec} with $r = 0$ (recalling that,
  for $d = 1$, the condition $s > d + 1$ may be dropped).
\end{proof}

We close this section by showing that Theorem~\ref{thm:abaci:rec} can be used
to compute $\boldsymbol{A}_{s, m, r}^d (q)$ for any choice of the parameters,
that is for any $s, m, d \geq 1$ and any integers $r$. First, note that,
if $0 \leq m s + r < s$, then $\boldsymbol{A}_{s, m, r}^d (q)
=\boldsymbol{A}_{m s + r, 1, 0}^d (q)$ (if $m s + r < 0$, then $\boldsymbol{A}_{s,
m, r}^d (q) = 1$). Second, if $m s + r \geq s$, we can write $m s + r =
m' s + r'$ where $m' \geq 1$ and $0 \leq r' < s$. By definition,
\begin{equation*}
  \boldsymbol{A}_{s, m, r}^d (q) =\boldsymbol{A}_{s, m', r'}^d (q) .
\end{equation*}
These observations reduce the general case to the case $s > r \geq 0$. In
light of Theorem~\ref{thm:abaci:rec}, the following initial conditions
therefore suffice to recursively determine $\boldsymbol{A}_{s, m, r}^d (q)$ for
any choice of the parameters. In the following, we use the notation $x^+ =
\max (0, x)$.

\begin{lemma}
  \label{lem:abaci:initial}Let $d, m \geq 1$ and $r$ be any integer. For
  $s \in \{ 1, 2, \ldots, d \}$,
  \begin{equation*}
    \boldsymbol{A}_{s, m, r}^d (q) = 1 + \min (s - 1, m s + r - 1)^+ q.
  \end{equation*}
  For the case $s = d + 1$, write $r = s r' + r_0$ with $0 \leq r_0
  \leq d$. If $m + r' < 0$, then $\boldsymbol{A}_{d + 1, m, r}^d (q) = 1$.
  If $m + r' \geq 0$, then
  \begin{equation*}
    \boldsymbol{A}_{d + 1, m, r}^d (q) = 1 + d (q + q^2 + \ldots + q^{m + r'})
     + (r_0 - 1)^+ q^{m + r' + 1} .
  \end{equation*}
\end{lemma}

\begin{proof}
  The case $s \in \{ 1, 2, \ldots, d \}$ is an immediate consequence of the
  fact that the corresponding abaci can have at most one bead.
  
  In the case $s = d + 1$, we can have a full column of beads in any of the $s
  - 1 = d$ columns after the first (and there can be no bead in a second
  column). Note that $\mathcal{A}_{d + 1, m, r}^d =\mathcal{A}_{d + 1, m + r',
  r_0}^d$. Since the case $m + r' < 0$ is trivial, suppose $m + r' \geq
  0$. Each of the last $d$ columns can accomodate $m + r'$ many beads.
  Additionally, any of the $r_0 - 1$ columns $1, 2, 3, \ldots, r_0 - 1$ can
  accomodate an additional bead in the $(m + r' + 1)$st row.
\end{proof}

\begin{example}
  In the case $d = 1$, we find that $\boldsymbol{A}_{1, m, r} (q) = 1$ and
  \begin{equation}
    \boldsymbol{A}_{2, m, r} (q) = 1 + q + q^2 + \ldots + q^{m + \lfloor r / 2
    \rfloor} . \label{eq:A:d1:s2}
  \end{equation}
  In the cases $r = \pm 1$, $q = 1$, these specialize to the initial
  conditions in Theorem~\ref{thm:cores:ns}.
\end{example}

\begin{example}
  Recall from \eqref{eq:CA} that, in the case $(d, r) = (1, - 1)$, we have
  $\boldsymbol{A}_{s, m, r} (q) =\boldsymbol{C}_{s, m, r} (q)$. For instance, for
  $(s, m) = (4, 3)$ we have $m s - 1 = 11$ and the $(4, 11)$-core partitions
  into distinct parts are
  \begin{eqnarray*}
    &  & \emptyset, \quad (1), \quad (2), \quad (3), \quad (2, 1), \quad (4,
    1), \quad (5, 2), \quad (6, 3), \quad (3, 2, 1),\\
    &  & (5, 2, 1), \quad (7, 4, 1), \quad (8, 5, 2), \quad (4, 3, 2, 1),
    \quad (6, 3, 2, 1), \quad (5, 4, 3, 2, 1) .
  \end{eqnarray*}
  There are $\boldsymbol{C}_{4, 3, - 1} (1) = 15$ such partitions and the
  generating polynomial for the number of parts in these partitions is
  \begin{equation*}
    \boldsymbol{C}_{4, 3, - 1} (q) = 1 + 3 q + 4 q^2 + 4 q^3 + 2 q^4 + q^5 = (1
     + q + q^2) (1 + 2 q + q^2 + q^3) .
  \end{equation*}
  We note that such a factorization always exists for the polynomials
  $\boldsymbol{C}_{s, m, - 1} (q)$. Indeed, for $s \geq 2$, we claim that
  $\boldsymbol{C}_{s, m, - 1} (q)$ is divisible by $1 + q + q^2 + \ldots + q^{m
  - 1}$. This is true for $\boldsymbol{C}_{2, m, - 1} (q)$ by \eqref{eq:A:d1:s2}
  and the claim follows inductively from Theorem~\ref{thm:cores:main}.
\end{example}

\begin{example}
  \label{eg:C:m2:rm1}Continuing the previous example in the special case $m =
  2$, we conclude from Theorem~\ref{thm:cores:main} that
  \begin{equation}
    \boldsymbol{C}_{s, 2, - 1} (q) = (1 + q)^{s - 1} . \label{eq:C:m2:rm1}
  \end{equation}
  In particular, the number of $(s, 2 s - 1)$-core partitions into $k$
  distinct parts is $\binom{s - 1}{k}$. Less generally, as observed in
  \cite{s-corepartitions}, there are $2^{s - 1}$ many $(s, 2 s - 1)$-core
  partitions into distinct parts.
\end{example}

\begin{remark}
  Example~\ref{eg:C:m2:rm1} suggests that there is a natural correspondence of
  compositions of $s$ and $(s, 2 s - 1)$-core partitions into distinct parts.
  Indeed, let us briefly describe a bijection between compositions $\mu$ of
  $s$ into $k$ parts and $(s, 2 s - 1)$-core partitions $\lambda$ into $k - 1$
  distinct parts. Given a partition $\lambda$, the corresponding composition
  $\mu = (\mu_1, \mu_2, \ldots, \mu_k)$ of $s$ is defined as follows. Consider
  the $s$-abacus of $\lambda$, and let $\mu_1 > 0$ be the number of empty
  columns before the first bead. Then we define $\mu_{j + 1}$, for $j = 1, 2,
  \ldots, k - 1$, depending on the location of the $j$th bead in the abacus.
  \begin{enumerate}
    \item If the $j$th bead is the only bead in its columns, then $\mu_{j +
    1}$ is $1$ more than the number of empty columns following that bead.
    
    \item If the $j$th bead is the first of two beads in its columns, then
    $\mu_{j + 1} = 1$.
    
    \item If the $j$th bead is the second of two beads in its columns, then
    $\mu_{j + 1}$ is the number of empty columns following that bead.
  \end{enumerate}
  Note that the size $\mu_1 + \ldots + \mu_k$ of the resulting composition
  $\mu$ is exactly $s$, the number of columns in the $s$-abacus of $\lambda$.
  We leave it to the interested reader to confirm that this correspondence is
  indeed a bijection.
\end{remark}

\section{Maximum number of parts and partitions with maximal initial
gaps}\label{sec:max}

The correspondence with abaci makes certain properties of $s$-core abaci
conveniently accessible. For instance, we can readily deduce the maximum
number of parts of the $s$-core partitions featured in
Theorem~\ref{thm:cores:main}.

\begin{lemma}
  \label{lem:maxnrparts}Let $M_{s, m, r}^d$ be the maximum number of parts of
  $(s, m s + r)$-core partitions into $d$-distinct parts. Let $d, m \geq
  1$. If $1 \leq r \leq d$ and $s > 1$, then
  \begin{equation}
    M_{s, m, r}^d = \left\lfloor \frac{s}{d + 1} \right\rfloor m + \left\{
    \begin{array}{ll}
      0, & \text{if $r = 1$ and $s \equiv 0, 1 \pmod{d +
      1}$},\\
      1, & \text{otherwise} .
    \end{array} \right. \label{eq:maxnrparts}
  \end{equation}
  Likewise, if $r = - 1$ and $s > 2$, then
  \begin{equation*}
    M_{s, m, - 1}^d = \left\lfloor \frac{s}{d + 1} \right\rfloor m + \left\{
     \begin{array}{ll}
       - 1, & \text{if $d = 1$ and $s \equiv 0 \pmod{2}$},\\
       0, & \text{otherwise, if $s \equiv 0, 1, 2 \pmod{d +
       1}$},\\
       1, & \text{otherwise} .
     \end{array} \right.
  \end{equation*}
\end{lemma}

\begin{proof}
  Observe that, by definition, $M_{s, m, r}^d = \deg (\boldsymbol{C}_{s, m, r}^d
  (q))$. The claim could therefore be deduced inductively from
  Theorem~\ref{thm:cores:main} together with the initial conditions in
  Lemma~\ref{lem:abaci:initial}.
  
  Instead, in Lemma~\ref{lem:maxnrparts:abaci}, we provide a combinatorial
  argument that explicitly constructs a partition with the maximum number of
  parts (or, equivalently, the corresponding $s$-core abacus). The second part
  of our claim, the case $r = - 1$, follows from using $M_{s, m, - 1}^d =
  M_{s, m - 1, s - 1}^d$ in Lemma~\ref{lem:maxnrparts:abaci}.
\end{proof}

Note that, by Lemma~\ref{lem:cores:abaci:smallr}, the definition of $M_{s, m,
r}^d$ in Lemma~\ref{lem:maxnrparts:abaci} agrees with the definition in
Lemma~\ref{lem:maxnrparts} if $1 \leq r \leq d$ or if $r = - 1$.

\begin{lemma}
  \label{lem:maxnrparts:abaci}Let $M_{s, m, r}^d$ be the maximum number of
  beads of an abacus in $\mathcal{A}_{s, m, r}^d$. Let $d, m \geq 1$. If
  $s \geq r = 1$, then \eqref{eq:maxnrparts} holds. If $s \geq r >
  1$, then
  \begin{equation}
    M_{s, m, r}^d = \left\lfloor \frac{s}{d + 1} \right\rfloor m +
    \left\lfloor \frac{r - 2}{d + 1} \right\rfloor + 1.
    \label{eq:maxnrparts:abaci}
  \end{equation}
\end{lemma}

\begin{proof}
  Since the maximum number of beads is always achieved by an abacus with a
  bead in position 1, we assume in the sequel that there is a bead in position
  $1$. Since there are $s - 2$ further positions in the first row, the maximum
  number of beads in the first row of an abacus in $\mathcal{A}_{s, m, r}^d$
  is
  \begin{equation}
    1 + \left\lfloor \frac{s - 2}{d + 1} \right\rfloor .
    \label{eq:maxnrparts:firstrow}
  \end{equation}
  Observe that each of the $m - 1$ subsequent rows can be filled with up to
  $\lfloor s / (d + 1) \rfloor$ many beads. Note that
  \eqref{eq:maxnrparts:firstrow} either equals $\lfloor s / (d + 1) \rfloor$
  or exceeds it by $1$. The former case happens precisely if $s \equiv 0, 1$
  modulo $d + 1$. This proves \eqref{eq:maxnrparts} in the case $r = 1$.
  
  In the case that \eqref{eq:maxnrparts:firstrow} equals $\lfloor s / (d + 1)
  \rfloor$, the $m + 1$st row may be filled with $1$ bead in position $m s +
  1$ plus an additional $\lfloor (r - 2) / (d + 1) \rfloor$ beads. In the case
  that \eqref{eq:maxnrparts:firstrow} exceeds $\lfloor s / (d + 1) \rfloor$ by
  $1$ (then the second column is empty except for a bead in the first row),
  the $m + 1$st row may be filled with $\lfloor (r - 2) / (d + 1) \rfloor$
  beads. In either case, we find that the maximum total number of beads is
  \eqref{eq:maxnrparts:abaci}.
\end{proof}

For comparison, let us consider $M_{s, t}$, the maximum number of parts of an
$(s, t)$-core partition (without any restriction on its parts). Equivalently,
$M_{s, t}$ is the largest part of an $(s, t)$-core partition. The following is
a corollary of Sylvester's theorem on the Frobenius problem; we refer to
\cite[Chapter~1]{beck-robins} for a beautiful exposition.

\begin{lemma}
  \label{lem:maxnrparts:st}For coprime $s, t > 1$, we have $M_{s, t} =
  \frac{1}{2} (s - 1) (t - 1)$.
\end{lemma}

\begin{proof}
  As discussed in Section~\ref{sec:abaci}, we can identify partitions
  $\lambda$ with the set $A (\lambda)$ of their first column hook lengths
  (these are precisely the positions of beads in the corresponding abaci).
  Note that $| A (\lambda) |$ is the number of parts of $\lambda$.
  
  By Lemma~\ref{lem:coreabacus}, $\lambda$ is $s$-core if and only if $s
  \leq h \in A (\lambda)$ implies that $h - s \in A (\lambda)$. It
  follows that, if $\lambda$ is $(s_1, \ldots, s_n)$-core, then $A (\lambda)$
  is a subset of the set $F (s_1, \ldots, s_n)$ of positive integers which
  cannot be written as a nonnegative linear combination of $s_1, \ldots, s_n$.
  Moreover, the set $F (s_1, \ldots, s_n)$ itself corresponds to a $(s_1,
  \ldots, s_n)$-core partition. It was shown by Sylvester
  \cite[Theorem~1.3]{beck-robins} that
  \begin{equation*}
    | F (s, t) | = \frac{(s - 1) (t - 1)}{2},
  \end{equation*}
  proving our claim.
\end{proof}

It is considerably more difficult to obtain the maximum size of core
partitions. In this direction, we only mention the following well-known result
of Olsson and Stanton \cite{os-core}.

\begin{theorem}
  Suppose $s$ and $t$ are coprime. The maximum size of an $(s, t)$-core
  partition is
  \begin{equation}
    \frac{(s^2 - 1) (t^2 - 1)}{24} . \label{eq:core:max}
  \end{equation}
\end{theorem}

Recall that $\mathcal{C}_{s, m, r}^d$ is the set of $(s, m s + r)$-core
partitions into $d$-distinct parts. The enumerations in
Theorem~\ref{thm:cores:main} are based on a recursive description in $s$ with
$m$ (as well as $r$ and $d$) fixed. Note that $\mathcal{C}_{s, m, r} \subseteq
\mathcal{C}_{s, m + 1, r}$. In the remainder of this section, we observe
another way in which $\mathcal{C}_{s, m, r}$ is naturally embedded in
$\mathcal{C}_{s, m + 1, r}$.

We say that the {\emph{initial gap}} of an integer partition $(\lambda_1,
\lambda_2, \ldots)$, with $\lambda_1 \geq \lambda_2 \geq \ldots$, is
the difference $\lambda_1 - \lambda_2$ of the first and second part. If a
partition $\lambda$ is $s$-core, then all the gaps between the parts of
$\lambda$ are strictly less than $s$. We say that a $(s_1, s_2, \ldots)$-core
partition $\lambda$ has maximal initial gap if $\lambda$ has initial gap $\min
(s_1, s_2, \ldots) - 1$.

Let $\mathcal{G}_{s, m, r}^d$ be the set of partitions $\lambda \in
\mathcal{C}_{s, m, r}^d$ with maximal initial gap, together with the empty
partition (in other words, $\mathcal{G}_{s, m, r}^d$ consists of those
$\lambda \in \mathcal{C}_{s, m, r}^d$ such that, if $\lambda$ has a first part
$\lambda_1$, then it has a second part $\lambda_2$ and $\lambda_1 - \lambda_2
= s - 1$).

\begin{lemma}
  Let $s, m \geq 1$. If $1 \leq r \leq d$ or $r = - 1$, then
  there is a natural $1$-$1$ correspondence between $(s, m s + r)$-core
  partitions into $d$-distinct parts with maximal initial gap and $(s, (m - 1)
  s + r)$-core partitions into $d$-distinct parts
\end{lemma}

\begin{proof}
  The claimed bijective correspondence $\mathcal{G}_{s, m, r}^d \rightarrow
  \mathcal{C}_{s, m - 1, r}^d$ is given by the map
  \begin{equation*}
    (\lambda_1, \lambda_2, \lambda_3, \ldots) \mapsto (\lambda_2, \lambda_3,
     \ldots),
  \end{equation*}
  with the understanding that the empty partition is sent to itself.
\end{proof}

\begin{corollary}
  In particular, if $1 \leq r \leq d$ or $r = - 1$, we get an
  enumeration of $(s, m s + r)$-core partitions into $d$-distinct parts with
  maximal initial gap from Theorem~\ref{thm:cores:main}.
\end{corollary}

\begin{example}
  The number of $(s, 3 s - 1)$-core partitions into $k$ distinct parts with
  maximal initial gap is $\binom{s - 1}{k - 1}$. This follows from the
  bijective correspondence between $\mathcal{G}_{s, 3, - 1}^1$ and
  $\mathcal{C}_{s, 2, - 1}^1$. Indeed, for the latter, as noted in
  \eqref{eq:C:m2:rm1}, the generating polynomial for the number of parts of
  $(s, 2 s - 1)$-core partitions into distinct parts is $\boldsymbol{C}_{s, 2, -
  1} (q) = (1 + q)^{s - 1}$. It was this example that initially motivated the
  present paper.
\end{example}

\section{Average number of parts}\label{sec:avg}

In this section, we indicate that our results on the enumeration of certain
$(s, t)$-core partitions can be used to obtain explicit formulas for the
averge number of parts of these partitions.

Determining the average size of core partitions has received considerable
attention in recent years. In particular, Johnson \cite{johnson-core} proved
the following result, which had been experimentally observed and conjectured
by Armstrong \cite[Conjecture~2.6]{ahj-cores} and which should be compared
with \eqref{eq:core:max}.

\begin{theorem}
  Suppose $s$ and $t$ are coprime. The average size of an $(s, t)$-core
  partition is
  \begin{equation}
    \frac{(s - 1) (t - 1) (s + t + 1)}{24} . \label{eq:core:avg}
  \end{equation}
\end{theorem}

The case $(s, t) = (s, s + 1)$ had been established by Stanley and Zanello
\cite{sz-core}. On the other hand, employing results by Ford, Mai and Sze
\cite{fms-cores} on self-conjugate core partitions, Chen, Huang and Wang
\cite{chw-core} showed that the average size of a self-conjugate $(s,
t)$-core partitions is also given by \eqref{eq:core:avg}.

Xiong \cite{xiong-cores} proved the following result for $(s, s + 1)$-core
partitions into distinct parts, which had been conjectured by Amdeberhan
\cite{amdeberhan-conj}.

\begin{theorem}
  The total sum of the sizes of $(s, s + 1)$-core partitions into distinct
  parts is
  \begin{equation}
    \sum_{\substack{
      i + j + k = s + 1\\
      i, j, k \geq 1
    }} F{}_i F_j F_k . \label{eq:core:fib:avg}
  \end{equation}
\end{theorem}

Note that the total sum \eqref{eq:core:fib:avg}, divided by the number $F_{s +
1}$ of such partitions, gives the average size of $(s, s + 1)$-core partitions
into distinct parts. Also note that, by general properties of constant
recursive sequences, the convolution sum \eqref{eq:core:fib:avg} of Fibonacci
numbers can expressed in the simpler but possibly less illuminating form
$\frac{1}{50} ((5 s + 7) s F_{s + 1} - 6 (s + 1) F_s)$. This result was
generalized by Zaleski \cite{zaleski-core} to higher moments of the sizes of
$(s, s + 1)$-core partitions into distinct parts. In subsequent work, Zaleski
\cite{zaleski-moments-d} also considered moments of the sizes of $(s, m s -
1)$-core partition into distinct parts. For recent results on the largest
sizes of $(s, m s \pm 1)$-core partitions into distinct parts, we refer to
\cite{xiong-max}.

As in Theorem~\ref{thm:cores:main}, let $f_s (q) =\boldsymbol{C}_{s, m, r}^d
(q)$ denote the generating polynomial for the number of parts of $(s, m s +
r)$-core partitions into $d$-distinct parts. Clearly, $n_s = f_s (1)$ is the
number of these partitions. On the other hand, note that $p_s = f_s' (1)$ is
the total number of parts of these partitions. We therefore obtain the
following result as a consequence of Theorem~\ref{thm:cores:main}.

\begin{corollary}
  Suppose the conditions of Theorem~\ref{thm:cores:main} hold. Then $p_s$, the
  total sum of the numbers of parts of $(s, m s + r)$-core partitions into
  $d$-distinct parts, satisfies
  \begin{equation*}
    p_s = p_{s - 1} + m p_{s - d - 1} + \binom{m + 1}{2} n_{s - d - 1} .
  \end{equation*}
\end{corollary}

Note that $p_s / n_s$ is the average number of parts of these partitions.
Since the $n_s$ are constant recursive by Theorem~\ref{thm:cores:main}, it
follows that the $p_s$ are constant recursive as well. We illustrate this in
the particularly pleasing cases of $(s, s + 1)$-core partitions into distinct
parts (in which case $n_s = F_{s + 1}$) as well as $(s, 2 s - 1)$-core
partitions into distinct parts (in which case $n_s = 2^{s - 1}$, as noted in
Example~\ref{eg:C:m2:rm1}).

\begin{corollary}
  The total sum of the numbers of parts of $(s, s + 1)$-core partitions into
  distinct parts is
  \begin{equation}
    \sum_{\substack{
      i + j = s\\
      i, j \geq 1
    }} F{}_i F_j = \frac{1}{5} (2 s F_{s + 1} - (s + 1) F_s) .
    \label{eq:core:fib:avg:parts}
  \end{equation}
\end{corollary}

Observe the surprising similarity of \eqref{eq:core:fib:avg:parts} with
\eqref{eq:core:fib:avg}, the total sum of the sizes of $(s, s + 1)$-core
partitions into distinct parts.

\begin{corollary}
  The total sum of the numbers of parts of $(s, 2 s - 1)$-core partitions into
  distinct parts is $(s - 1) 2^{s - 2}$. The number of these partitions is
  $2^{s - 1}$.
\end{corollary}

Likewise, the higher moments for the number of parts of $(s, m s + r)$-core
partitions into $d$-distinct parts can be obtained from higher derivatives of
the generating polynomial $f_s (q) =\boldsymbol{C}_{s, m, r}^d (q)$.

\section{Conclusion}

Under certain restrictions on $s$ and $t$, we have determined the generating
polynomials $f_{s, t} (q)$ for $(s, t)$-core partitions into $d$-distinct
parts, where $q$ keeps track of the number of parts. It would be desirable,
but appears difficult, to prove enumeration results for general $s$ and $t$.
It would also be desirable, but again appears difficult, to be able to
incorporate the sizes of the partitions in the generating polynomials. More
generally, it is natural to ask whether we can include additional statistics
(besides the number of parts) to the enumeration results for core partitions
(see, for instance, the final section of \cite{paramonov-core} for a
discussion of including the bounce statistic on $(s, m s + 1)$-core
partitions).

The natural question whether Anderson's \cite{anderson-cores} result
\eqref{eq:anderson} that the number of $(s, t)$-core partitions is given by
the generalized Catalan numbers can be similarly extended to keep track of
additional statistics, like the number of parts, leads to intriguing and
surprisingly difficult open problems. For instance, it is an open problem
(proposed by Dennis Stanton; see \cite{ahj-cores}) to find a statistic
$\operatorname{stat} (\lambda)$ on $(s, t)$-core partitions $\lambda$ such that, for
coprime $s$ and $t$,
\begin{equation}
  \sum_{\lambda} q^{\operatorname{stat} (\lambda)} = \frac{1}{[s + t]_q} \binom{s +
  t}{s}_q, \label{eq:anderson:q}
\end{equation}
where the sum is over all $(s, t)$-core partitions $\lambda$ and where the
right-hand side are the usual $q$-analogs. Interestingly, Armstrong, Hanusa
and Jones \cite[Conjecture~2.8]{ahj-cores} provide a conjectural candidate
for such a statistic, which is the sum of $n (\lambda)$, the number of parts
of $\lambda$, plus a second (nonnegative) statistic $m (\lambda)$, the skew
length of $\lambda$. Anderson established \eqref{eq:anderson} by exhibiting a
bijection between the set of $(s, t)$-core partitions and the set of $(s, t)$
Dyck paths (lattice paths from one corner of a $s \times t$ rectangle to the
opposite corner, which stay above the diagonal connecting these corners). In
particular, statistics on core partitions can be obtained from statistics on
Dyck paths. As an indication that proving \eqref{eq:anderson:q}, a $q$-analog
of Anderson's enumeration \eqref{eq:anderson}, is a difficult problem, we
mention that even the nonnegativity of the coefficients of the generalized
$q$-Catalan numbers on the right-hand side of \eqref{eq:anderson:q} is
nontrivial and that no elementary proof is known. On the other hand, in the
words of \cite{ahj-cores}, the conjecture that \eqref{eq:anderson:q} holds
for $\operatorname{stat} (\lambda) = n (\lambda) + m (\lambda)$ is ``just a shadow
from the more general subject of $q, t$-Catalan combinatorics'' (below, we
will use the letter $x$ in place of $t$). A particularly appealing general
conjecture is that
\begin{equation}
  \sum_{\lambda} q^{n (\lambda)} x^{m (\lambda)} = \sum_{\lambda} q^{m
  (\lambda)} x^{n (\lambda)}, \label{eq:catalan:qt:sym}
\end{equation}
where both sums are over $(s, t)$-core partitions $\lambda$ for coprime $s$
and $t$. Here, the left-hand side defines the $q, x$-Catalan numbers. For more
details and further references, we refer to \cite{ahj-cores}, as well as
\cite{gorsky-mazin-catalan-2} where special cases of the symmetry
\eqref{eq:catalan:qt:sym} are proved.

As noted in Example~\ref{eg:C:m2:rm1}, it follows from $\boldsymbol{C}_{s, 2, -
1} (q) = (1 + q)^{s - 1}$ that the number of $(s, 2 s - 1)$-core partitions
into $k$ distinct parts is $\binom{s - 1}{k}$. Less generally, as observed in
\cite{s-corepartitions}, there are $2^{s - 1}$ many $(s, 2 s - 1)$-core
partitions into distinct parts. On the other hand, Yan, Qin, Jin, Zhou
\cite{yqjz-core}, Zaleski, Zeilberger \cite{zz-core}, Baek, Nam, Yu
\cite{bny-core}, and Paramonov \cite{paramonov-core} show that, for odd
$s$, the number of $(s, s + 2)$-core partitions into distinct parts is $2^{s -
1}$ as well. In our notation~\eqref{eq:C:def},
\begin{equation}
  \boldsymbol{C}_{s, 1, 2} (1) = 2^{s - 1} . \label{eq:C:2s}
\end{equation}
It would be interesting to determine an explicit formula for $\boldsymbol{C}_{s,
1, 2} (q) = \sum c_n (s) q^n$, a $q$-analog of \eqref{eq:C:2s}, and thus
obtain the number of such partitions into $k$ parts. Limited numerical data
suggests that $c_n (s)$, for $s \geq n$, is a polynomial in $s$ of degree
$n$.

\begin{acknowledgements}
We thank Huan Xiong for sending a preprint of
\cite{xiong-max} and for mentioning Zaleski's conjectures
\cite{zaleski-moments-d} on the moments of the size of an $(n, d n -
1)$-core partition into distinct parts.
\end{acknowledgements}


\end{document}